\documentclass[12pt]{amsart}

\usepackage[margin=1.15in]{geometry}
\usepackage{amscd,amssymb, amsmath, wasysym, mathrsfs, mathtools, bbold, color}
\usepackage{graphicx}
\usepackage[all, cmtip]{xy}

\usepackage{url}

\definecolor{hot}{RGB}{65,105,225}

\usepackage[pagebackref=true,colorlinks=true, linkcolor=hot ,  citecolor=hot, urlcolor=hot]{hyperref}

\theoremstyle{plain}
\newtheorem{theorem}{Theorem}[section]
\newtheorem{prop}[theorem]{Proposition}

\newtheorem*{claim}{Claim}

\newtheorem{cor}[theorem]{Corollary}
\newtheorem{conj}[theorem]{Conjecture}
\newtheorem{lemma}[theorem]{Lemma}

\theoremstyle{definition}

\newtheorem{defn}[theorem]{Definition}
\newtheorem{rmk}[theorem]{Remark}

\newtheorem*{ex*}{Example}

\newtheorem*{question}{Question}

\newcommand\sA{{\mathcal A}}
\newcommand\sO{{\mathcal O}}

\newcommand\sF{{\mathcal F}}

\newcommand\sR{\mathcal{R}}
\newcommand\rh{\mathcal{RH}}

\newcommand\fV{\mathscr{V}}
\newcommand\fW{\mathscr{W}}

\newcommand\pp{{\mathbf{P}}}
\newcommand\bP{{\mathbf{P}}}

\newcommand\bQ{{\mathbf{Q}}}
\newcommand\zz{{\mathbf{Z}}}

\newcommand\bR{{\mathbf{R}}}
\newcommand\cc{{\mathbf{C}}}
\newcommand\bC{{\mathbf{C}}}

\newcommand\aaa{{\mathbf{A}}}

\newcommand\hh{{\mathbf{H}}}

\newcommand\mb{\mathbf{M}_{\textup{B}}}
\newcommand\mdr{\mathbf{M}_{\textup{DR}}}

\newcommand\mhod{\mathbf{M}_{\textup{Hod}}}
\newcommand\odr{\Omega_{\textup{DR}}}

\newcommand\betti{\textup{B}}
\newcommand\dr{\textup{DR}}
\newcommand\res{\rm{res}}
\newcommand\ev{\rm{ev}}

\newcommand{\ubul}{{\,\begin{picture}(-1,1)(-1,-3)\circle*{2}\end{picture}\ }}

\DeclareMathOperator{\id}{id}                    

\DeclareMathOperator{\im}{Im}

\DeclareMathOperator{\pic}{Pic}
\DeclareMathOperator{\rank}{rank}

\DeclareMathOperator{\homo}{Hom}

\title{Cohomology jump loci of quasi-compact K\"ahler manifolds}

\author{Nero Budur}
\address{KU Leuven, Celestijnenlaan 200B, B-3001 Leuven, Belgium} 
\email{Nero.Budur@kuleuven.be}

\author{Botong Wang}
\address{University of Wisconsin, Van Vleck Hall, 480 Lincoln Drive, Madison, WI, USA}
\email{bwang274@wisc.edu}

\begin{document}

\date{}

\begin{abstract}
We give two applications of the exponential Ax-Lindemann Theorem to local systems. One application is to show that for a connected topological space, the existence of a finite model of the real homotopy type implies linearity of the cohomology jump loci around the trivial local system. Another application is the linearity of the cohomology jump loci of rank one local systems on quasi-compact K\"ahler manifolds. 
\end{abstract}
\maketitle
\section{Introduction}
Given a path-connected topological space $X$, we define the \textbf{Betti moduli space} 
$$\mb(X)=\homo(\pi_1(X), \cc^*)$$
to be the space of rank one characters of $\pi_1(X)$. When $X$ is of the homotopy type of a finite CW-complex, $\mb(X)$ is an algebraic group, and it is isomorphic to the direct product of $(\cc^*)^{b_1(X)}$ and a finite group.

For any representation $\rho\in\mb(X)$, there is a unique $\cc$-local system of rank one $L_\rho$ on $X$, whose monodromy representation is isomorphic to $\rho$. Thus, $\mb(X)$ is naturally the moduli space of rank one local systems on $X$. In $\mb(X)$, there are some canonically defined algebraic subsets,
$$\Sigma^i_k(X)=\{\rho\in \mb(X)| \dim H^i(X, L_\rho)\geq k\},$$
called the \textbf{cohomology jump loci}. One can define these loci more precisely as subschemes of $\mb(W)$, not necessarily reduced. However, for the purpose of this paper, we only consider their induced reduced structures. 

In this article, a {\bf subtorus} is an algebraic subgroup $(\cc^*)^p\subset \mb(X)$.
Based on the work of Green and Lazarsfeld \cite{gl}, Arapura \cite{a} proved that when $X$ is a compact K\"ahler manifold, each $\Sigma^i_k(X)$ is a finite union of translated  subtori in $\mb(X)$.  When $X$ is a quasi-compact K\"ahler manifold (the complement of a normal crossing divisor in a compact K\"ahler manifold), Arapura proved the same statement holds when the mixed Hodge structure on $H^1(X, \cc)$  is pure. When $X$ is a projective manifold, Simpson \cite{si} showed that $\Sigma^i_k(X)$ is a finite union of torsion-translated subtori. When $X$ is a quasi-projective manifold, it was first proved by Dimca and Papadima \cite{dp} that any irreducible component of $\Sigma^i_k(X)$ passing through the trivial representation $\mathbf{1}$ is a subtorus. Later, in \cite{bw1} we proved in this case that each irreducible component of $\Sigma^i_k(X)$ is a torsion translate of a subtorus in $\mb(X)$. Around the same time, the second author proved in \cite{w} the same result for all compact K\"ahler manifolds, achieving a final positive answer to an original conjecture by Beauville \cite{Be}. See the survey \cite{bw-sur} for more on this subject.


The first main result is a generalization of Arapura's result by removing the assumption on $H^1(X, \cc)$. 

\begin{theorem}\label{main1}
Let $X$ be a connected quasi-compact K\"ahler manifold. Then each $\Sigma^i_k(X)$ is a finite union of translates of subtori in $\mb(X)$. 
\end{theorem}
As mentioned above, when $X$ is a quasi-projective manifold this was already proved in \cite{bw1}. However, our proof here is quite different from the one of \cite{bw1}. We generalize Simpson's (see \cite{si}) Betti-de Rham sets of projective manifolds to quasi-compact K\"ahler manifolds, and then prove the irreducible ones are translates of subtori. In particular, the cohomology jump loci are Betti-de Rham sets, and hence are union of translates of subtori. 

We conjecture that the full non-compact analog of the main result of \cite{w} holds: 
\begin{conj}\label{conj}
Let $X$ be a connected quasi-compact K\"ahler manifold. Then each irreducible component of $\Sigma^i_k(X)$ contains a torsion point. 
\end{conj}

Moreover, to the best of our knowledge, we do not know an example of a quasi-compact K\"ahler manifold that is not homotopy equivalent to any quasi-projective manifold. In the compact case, such an example was constructed by Voisin \cite{v}. So an interesting question would be the following. 

\begin{question}
Is there a quasi-compact K\"ahler manifold that is not homotopy equivalent to any quasi-projective manifold?
\end{question}

The second main result reveals more of the topology responsible for the linearity phenomenon for cohomology jump loci. In this article we use the shorter terminology {\bf differential algebra} to mean a commutative differential graded algebra over $\cc$.

\begin{theorem}\label{main2}
Suppose $X$ is a path-connected topological space, and suppose $X$ is homotopy equivalent to a finite CW-complex.
\begin{enumerate}
\item If the real homotopy type of $X$ is equivalent to a finite-dimensional differential algebra, then each irreducible component of $\Sigma^i_k(X)$ passing through the origin is a subtorus. 
\item If $\rho\in \mb(X)$ and the differential algebra pair $(\odr^\ubul(X), \odr^\ubul(L_\rho))$ is homotopy equivalent to another differential algebra pair $(C^\ubul, M^\ubul)$ such that $M^\ubul$ has finite dimension on each degree, then every irreducible component of $\Sigma^i_k(X)$ passing through $\rho$ is the translate of a subtorus. 
\end{enumerate}
\end{theorem}

In part (2), $\odr^\ubul(X)$ is Sullivan's differential algebra of piecewise smooth $\bC$-forms on  $X$.  The differential algebra $\odr^\ubul(X)$ can be replaced by the de Rham complex of smooth $\bC$-forms on $X$ if $X$ is a smooth manifold. In the above, $\odr^\ubul(L_\rho)$ is the module over $\odr^\ubul(X)$  obtained from the forms with values in the rank one local system $L_\rho$ attached to $\rho$. Also in part (2), the notions of pair and homotopy between pairs are as defined in \cite{bw2}. Part (2) implies part (1) if $\rho=\mathbf{1}$. Part (1) is a generalization of a result of Dimca-Papadima \cite[Theorem C(2)]{dp} which required in addition a positive-weight structure on the finite-dimensional model. Our proof is different, not depending on weight structures. 

Recall that, by a theorem of J. Morgan \cite{m}, the real homotopy type of a quasi-compact K\"ahler manifold is equivalent to the Gysin model with respect to any good compactification, and the Gysin model is finite-dimensional. Nilmanifolds and solvmanifolds also admit finite-dimensional models for the real homotopy types. However, the cohomology jump loci of local systems of rank one are trivial and, respectively, finite in these cases \cite{MP, Mi}. It would be interesting to find other classes of manifolds with non-trivial jump loci and admitting finite models for their real homotopy types. 

{{One can construct examples as in part (2) with $\rho$ not trivial as follows. Take $\rho$ with finite image. Then $\rho$ is trivialized up on a finite abelian Galois cover $X'$ of $X$. Assuming that $X'$ satisfies the condition from part (1), then $\rho$ satisfies the condition from part (2).}}

The above two theorems are both consequences of the following  fact. 

\begin{prop}\label{main0}
Let $\exp: \cc^n\to (\cc^*)^n, (z_i)_{1\leq i\leq n}\mapsto (e^{2\pi iz_i})_{1\leq i\leq n}$ be the exponential map. Suppose $V\subset \cc^n$ and $W\subset (\cc^*)^n$ are irreducible algebraic subvarieties of the same dimension. If $\exp(V)\subset W$, then $V$ is a translate of a linear subspace, and hence $W$ is a translate of a subtorus. 
\end{prop}

This proposition is an affine analog of \cite[Theorem 1.3(c)]{si}. In fact, when $X$ has a compactification $\bar{X}$ with $H_1(\bar{X}, \zz)=0$ and complement a simple normal crossings divisor, the de Rham moduli space of $(\bar{X}, \bar{X}\setminus X)$ is isomorphic to $\cc^{b_1(X)}$ and the Betti moduli space of $X$ is isomorphic to $(\cc^*)^n$, and the Riemann-Hilbert map is isomorphic to the above exponential map. More precisely, we have the following commutative diagram
$$
\xymatrix{
\cc^{b_1(X)}\ar[r]^-{\cong}\ar[d]^{\exp}&\mdr(\bar{X}/D)\ar[d]^-{RH}\\
(\cc^*)^{b_1(x)}\ar[r]^-{\cong}&\mb(X)
}
$$
where $D=\bar{X}\setminus X$, $\mdr(\bar{X}/D)$ and $\mb(X)$ are the de Rham and Betti moduli spaces, respectively, and $RH$ is the Riemann-Hilbert map. See Section \ref{Kahler} for more details. 

Proposition \ref{main0} is known among the experts. It is implied by the Exponential Ax-Lindemann Theorem due to Ax \cite{Ax}, see \cite[Theorem 1.2]{cll}. A related statement was proved in \cite[Appendix]{bw1} and used there for the main result. Nevertheless, we decided to provide a proof in Section 2, since the proof is quite short and simple, and some of the ideas are used in the last section. 

Following the work of Dimca-Papadima \cite{dp}, it is straightforward to prove Theorem \ref{main2} using Proposition \ref{main0}. We give the argument in Section 2. In Section 3, we prove Theorem \ref{main1}. The proof goes through a non-compact analog of \cite[Theorem 1.3(c)]{si}. More precisely, we prove that Betti-de Rham sets of a 1-Hodge structure are finite unions of translates of subtori. Then we show the cohomology jump loci are Betti-de Rham sets.  We not only need the statement of \cite[Theorem 1.3(c)]{si} in the proof of Theorem \ref{main1}, but also use the ideas in \cite{si} repeatedly throughout the paper. 

Note that, according to Theorem \ref{main2}, we give an alternative proof of Theorem \ref{main1} by showing that every irreducible component of $\Sigma^i_k(X)$ contains a local system  admitting $L$ a finite-dimensional model for its twisted de Rham complex. If $L$ is torsion and $X$ is quasi-compact K\"ahler manifold, then a finite model always exists, by Morgan's theorem  plus a standard argument with finite covers reducing to the case of the trivial rank one local system (see e.g. the last paragraph in the proof of \cite[Theorem 1.3]{w}). However, as noted in Conjecture \ref{conj}, we cannot yet  produce torsion points.

\medskip
\noindent
{\it Acknowledgement.} The first author was sponsored by  FWO, KU Leuven OT, and Methusalem grants.

\section{Affine Betti-de Rham sets}
Before the proof of Proposition \ref{main0}, we give the definition of some terms that we already used in the introduction. Let $\exp: \cc^n\to (\cc^*)^n, (z_i)_{1\leq i\leq n}\mapsto (e^{2\pi iz_i})_{1\leq i\leq n}$ be the exponential map. Let $V\subset\cc^n$ be an irreducible algebraic subvariety. If there exists an algebraic subvariety $W\subset(\cc^*)^n$ such that $\dim(V)=\dim(W)$ and $\exp(V)\subset W$, then $W$ is called an \textbf{irreducible affine Betti-de Rham set}. An analytic subset of $(\cc^*)^n$ is called an \textbf{affine Betti-de Rham set} if it is a finite union of irreducible Betti-de Rham sets of the same dimension. 

Thus, Proposition \ref{main0} is equivalent to the statement that every irreducible affine Betti-de Rham set in $(\cc^*)^n$ is the translate of a subtorus. 

\begin{lemma}\label{cover}
Let $V\subset \cc^n$ and $W\subset (\cc^*)^n$ be closed irreducible analytic varieties of the same dimension. Suppose $\exp(V)\subset W$. Then $\exp(V)=W$. 
\end{lemma}
\begin{proof}
Let $\widetilde{W}=\exp^{-1}(W)$. Then $\exp: \widetilde{W}\to W$ is a covering map. Since $W\subset (\cc^*)^n$ is a closed analytic subvariety, $\widetilde{W}$ is a closed analytic subset consisting of possibly infinitely many irreducible analytic varieties. Let $W_{\text{reg}}$ be the smooth locus of $W$, and let $\widetilde{W}_{\text{reg}}=\exp^{-1}(W_{\text{reg}})$. Since irreducible components of an analytic variety correspond to the connected components of the smooth locus of the analytic variety, the irreducible components of $\widetilde{W}$ are parametrized by the cokernel of  the homomorphism $H_1(W_{\text{reg}}, \zz)\to H_1((\cc^*)^n, \zz)$, induced by the composition of embeddings $W_{\text{reg}}\hookrightarrow W\hookrightarrow (\cc^*)^n$. Moreover, any two irreducible components of $\widetilde{W}$ differ by a translation by a lattice point in $\cc^n$. Thus, for any irreducible components $\widetilde{W}_0$ and $\widetilde{W}_0'$, we have $\exp(\widetilde{W}_0)=\exp(\widetilde{W}_0')$, and hence, $\exp(\widetilde{W}_0)=W$ for every irreducible component $\widetilde{W}_0$ of $\widetilde{W}$.

Since $\exp(V)\subset W$, one  has $V\subset \widetilde{W}$. Hence, $V$ is contained in one of the irreducible components of $\widetilde{W}$, which we denote by $\widetilde{W}_1$. Now, $V\subset \widetilde{W}_1$. Moreover, $V$ and $\widetilde{W}_1$ are both closed irreducible analytic varieties of the same dimension. Therefore $V=\widetilde{W}_1$, and hence the lemma follows. 
\end{proof}

\begin{proof}[Proof of Proposition \ref{main0}]
The proposition is trivial for $n=1$. We first prove the proposition for the case $n=2$. Then we will prove the proposition for general $n$ by induction. 

When $n=2$, the only nontrivial case is when $V$ is of dimension one. Let $\widetilde{W}=\exp^{-1}(W)$. According to Lemma \ref{cover}, $\widetilde{W}$ is a  union of possibly infinitely many closed irreducible analytic varieties, and $V$ is equal to one irreducible component of $\widetilde{W}$. Let $f: H_1(W_{\text{reg}}, \zz)\to H_1((\cc^*)^2, \zz)$ be the map on homology induced by the embedding $W_{\text{reg}}\hookrightarrow (\cc^*)^2$. Then $\im(f)$ acts on each of the irreducible components of $\widetilde{W}$, and hence on $V$.

\begin{claim}
$\im(f)$ has rank at least one. 
\end{claim}
\begin{proof}[Proof of Claim]
We can consider $(\cc^*)^2$ as $\pp^2$ removing three lines $L_1$, $L_2$ and $L_3$. Let $\bar{V}$ be the closure of $V$ in $\pp^2$. Clearly, $\bar{V}$ intersects each of $L_1$. Choose a point $x\in \bar{V}\cap L_1$. Take a small loop $\gamma$ in $V$ around $x$. (When $x$ is a singular point, we can still make such a choice, but no longer unique up to homotopy.) Since $\gamma$ represents a nontrivial loop in $H_1((\cc^*)^2, \zz)$, $f([\gamma])\in H_1((\cc^*)^2, \zz)$ is non-zero. Since $H_1((\cc^*)^2, \zz)$ is torsion free, $\im(f)$ has rank at least one. 
\end{proof}
The group $H_1((\cc^*)^2, \zz)=\pi_1((\cc^*)^2)$ acts on $\cc^2$ by translations. Now, the subgroup $\im(f)$ acts on $\cc^2$, which preserves $V$. Since $\im(f)$ has rank at least one, $V$ contains infinitely many points on an affine line. Since $V$ is a closed irreducible algebraic variety of dimension one, $V$ must be equal to that affine line. Hence $W$ must be a translate of subtorus. 

Next, we prove the proposition by induction on $n$ ($n\geq 3$). Suppose $W$ is not of codimension one, or equivalently $V$ is not of codimension one. Take a general projection of algebraic groups $p: (\cc^*)^n\to (\cc^*)^{n-1}$, such that $\dim(W)=\dim(p(W))$. Let the induced map on universal cover be $\tilde{p}: \cc^n\to \cc^{n-1}$. By Lemma \ref{cover}, $\exp(V)=W$, and hence $\exp(\tilde{p}(V))=p(W)$. 
Since the exponential map is continuous, $\exp\left(\overline{\tilde{p}(V)}\right)\subset \overline{p(W)}$. Notice that $\overline{p(W)}$ is a closed algebraic variety, because of the hypothesis that $W$ is algebraic. Clearly, $\exp(\tilde{p}(V))$ is Zariski dense in $\exp\left(\overline{\tilde{p}(V)}\right)$, and $p(W)$ is Zariski dense in $\overline{p(W)}$. Thus, 
$$\dim \exp\left(\overline{\tilde{p}(V)}\right)=\dim \exp(\tilde{p}(V))=\dim p(W)=\dim \overline{p(W)}.$$

We obtain irreducible subvarieties $\overline{\tilde{p}(V)}\subset \cc^{n-1}$ and $\overline{p(W)}\subset (\cc^*)^{n-1}$ of the same dimension. Moreover, $\exp\left(\overline{\tilde{p}(V)}\right)\subset \overline{p(W)}$. Therefore, the induction hypothesis implies that $\overline{\tilde{p}(V)}$ is a translate of a linear subspace. This is true for any projection $p: (\cc^*)^n\to (\cc^*)^{n-1}$ such that $\dim p(W)=\dim W$. Taking two distinct such projections, then $V$ is contained in two translates of linear subspaces $V'$ and $V''$ with $\dim V'=\dim V''=\dim V+1$. Therefore, $V=V'\cap V''$ has to be a translate of a linear subspace in $\cc^n$. 


Suppose $W$ is of codimension one. Take any affine embedding $j: (\cc^*)^{n-1}\to (\cc^*)^n$; that is, $j$ is the composition of an algebraic group embedding and a translation in $(\cc^*)^n$. Let $T=\im(j)$ and let $S=\exp^{-1}(T)$. Then $S$ is the union of infinitely many parallel hyperplanes in $\cc^n$. By Lemma \ref{cover}, $\exp(V)=W$. Hence $\exp(S\cap V)=T\cap W$. Suppose $T$ intersects $W$ properly, i.e., $T\cap W\neq \emptyset$ and $T\nsubseteq W$. Then $T\cap W$ is of dimension $n-2$. Notice that $S\cap V$ consists of at most countably many closed subvarieties of $V$, and $\exp(S\cap V)=T\cap W$. Since an analytic variety can not be covered by countably many subvarieties of smaller dimension, there exists an irreducible component of $S\cap V$ of dimension $n-2$, which we denote by $V'$. Since $V'$ is irreducible, $\exp(V')$ is contained in one of the irreducible components of $T\cap W$, which we denote by $W'$. Since $\dim(V')=\dim(W')=n-2$, by Lemma \ref{cover}, $\exp(V')=W'$. Then by induction hypothesis, $W'$ is a translate of a subtorus and $V'$ is an affine subspace of $\cc^n$. 

Since $V$ is irreducible and since $\dim V\geq 2$, for a general affine hyperplane $H\subset \cc^n$, $V\cap H$ is irreducible. Therefore, for a general affine hyperplane $H$ in $\cc^n$, whose slope is rational, $H\cap V$ is an affine subspace of $\cc^n$. Thus, $V$ has to be an affine subspace of $\cc^n$, and $W$ is a translate of a subtorus in $(\cc^*)^n$. 
\end{proof}



\begin{cor}\label{germtori}
Suppose $(\fV, 0)$ and $(\fW, 1)$ are analytic germs of two algebraic sets in $\cc^n$ and $(\cc^*)^n$, respectively. If the exponential map $\exp: \cc^n\to (\cc^*)^n$ induces an isomorphism between $(\fV, 0)$ and $(\fW, 1)$, then $(\fW, 1)$ is the germ of a finite union of subtori. 
\end{cor}
\begin{proof}
Let $V$ and $W$ be the analytic closure of $(\fV, 0)$ and $(\fW, 1)$ in $\cc^n$ and $(\cc^*)^n$, respectively. Since $(\fV, 0)$ and $(\fW, 1)$ are germs of algebraic sets, $V$ and $W$ are algebraic sets of $\cc^n$ and $(\cc^*)^n$, respectively. Moreover, $(\fV, 0)$ is equal to the germ of $V$ at $0$ and similarly $(\fW, 1)$ is equal to the germ of $W$ at $1$. Since $\exp(\fV)=\fW$, we deduce that $\fW\subset W$, and since $W$ is analytically closed, we have $\exp(V)\subset W$. 

Let $W_1$ be an irreducible component of $W$. Since $W$ is the analytic closure of $\fW$, $W_1\cap \fW$ is nonempty and contains an analytic open subset of $W_1$. Since $\exp(\fV)=\fW$, there exists an irreducible component of $V$, denoted by $V_1$, such that $\exp(V_1)$ contains a nonempty analytic open subset of $W_1$. In particular, $\dim V_1\geq \dim W_1$. Since $V_1$ is irreducible, $\exp(V_1)$ is also irreducible. Notice that $\exp(V_1)\subset \exp(V)\subset W$. Hence $\exp(V_1)$ is contained in one of the irreducible components of $W$, which is clearly $W_1$. Thus, we have shown $\exp(V_1)\subset W_1$ and $\dim V_1= \dim W_1$.

Now, we have algebraic subvarieties $V_1\subset \cc^n$ and $W_1\subset (\cc^*)^n$ of the same dimension, and $\exp(V_1)\subset\exp(W_1)$. Therefore, by Proposition \ref{main0}, $W_1$ is a torus. Since $W_1$ can be chosen as any irreducible component of $W$, the variety $W$ is a union of subtori, and $\fW$ is the germ of a finite union of subtori. 
\end{proof}

Before proving Theorem \ref{main2}, recall that when $X$ is a connected smooth manifold, $\mb(X)$ is  isomorphic to the space of  complex flat line bundles on $X$. There is an exponential map from $H^1(X, \cc)$ to the space of flat connections, mapping the origin to the trivial flat connection. Thus, there exists a complex Lie group map from $H^1(X, \cc)$ to $\mb(X)$, which we call the \textbf{exponential map} and denote by $\exp$. 

We also need some definitions and results from \cite{dp}. A differential algebra $(\sA^\ubul, d)$ over $\cc$ is called connected if $\sA^0=\cc\cdot\id$. Suppose $(\sA^\ubul, d)$ is connected and suppose $\sA^i$ is finite-dimensional for every $i$. The \textbf{space of flat connections} is defined to be 
$$\sF(\sA^\ubul)=\{\omega\in \sA^1| d\omega=0\}.$$
Since $\sA^\ubul$ is connected, the first differential $d^0: \sA^0\to \sA^1$ is zero, and hence $\sF(\sA^\ubul)\cong H^1(\sA^\ubul, d)$. We will consider $\sF(\sA^\ubul)$ as an affine variety with an origin, instead of a vector space. Given an element $\omega\in \sF(\sA^\ubul)$, the operator $d_\omega\stackrel{\textrm{def}}{=}d+\omega$ defines another differential on the graded vector space $\sA^\ubul$, since $(d+\omega)^2=0$. We call the cochain complex $(\sA^\ubul, d_\omega)$ the \textbf{Aomoto complex associated to} $\omega$. 

The \textbf{ resonance variety} of $\sA^\ubul$ is defined to be
$$\sR^i_k(\sA^\ubul)=\{\omega\in \sF(\sA^\ubul)| \dim H^i(\sA^\ubul, d_\omega)\geq k\}. $$
The  resonance varieties are closed algebraic subsets of $\sF(\sA^\ubul)$. In fact, they are defined by certain determinantal ideals, see loc. cit. or \cite{bw2}. 

\begin{theorem}{\cite[Theorem B(2)]{dp}}\label{germ}
Suppose $X$ is a connected topological space, homotopy equivalent to a finite CW-complex. Suppose $\odr^\ubul(X)$ is homotopy equivalent to a connected  differential  algebra $\sA^\ubul$, with $\dim \sA^i<\infty$ for every $i$. Then the composition $\exp\circ \phi: \sF(\sA^\ubul)\to \mb(X)$ induces an isomorphism on the analytic germ of $\sR_k^i(\sA^\ubul)$ at the origin and the analytic germ of $\Sigma^i_k(X)$ at the trivial character $\mathbf{1}$, where $\exp: H^1(X, \cc)\to \mb(X)$ is the exponential map and $\phi: \sF(\sA^\ubul)\to H^1(X, \cc)$ is the isomorphism induced by the homotopy equivalence between $\sA^\ubul$ and $\odr^\ubul(X)$.
\end{theorem}

Here $\odr^\ubul(X)$ is, as in the Introduction, Sullivan's differential algebra of piecewise smooth $\bC$-forms on  $X$. It  can be replaced by the de Rham complex of $X$ if $X$ is a smooth manifold. 

\begin{proof}[Proof of part (1) of Theorem \ref{main2}]
As in the above theorem, the composition $\exp\circ \phi: \sF(\sA^\ubul)\to \mb(X)$ has image in the connected component of $\mb(X)$ containing $\mathbf{1}$, and it is isomorphic to the exponential map $\exp: \cc^{b_1(X)}\to (\cc^*)^{b_1(X)}$, if we replace $\mb(X)$ by that connected component. Therefore, the theorem follows immediately from Theorem \ref{germ} and Corollary \ref{germtori}. 
\end{proof}

The theory of differential graded Lie algebra pairs in \cite{bw2} (or in the case of rank one flat bundles, commutative differential graded algebra pairs) allows us to move away from origin. To study the deformation theory of the cohomology jump loci at a general point $\rho\in \mb(X)$, we need not only the differential algebra $\odr^\ubul(X)$, but also $\odr^\ubul(L_\rho)$ considered as a module over $\odr^\ubul(X)$. This allows us to give a generalized version of Theorem \ref{main2} as below, using the definition of homotopy equivalence of pairs from \cite{bw2}.

\begin{proof}[Proof of part (2) of Theorem \ref{main2}]
The proof is same as the proof of part (1) of Theorem \ref{main2}, except we need to replace Theorem \ref{germ} by Theorem 7.2 of \cite{bw2}. Note that in \cite{bw2} we worked only with smooth manifolds instead of CW-complexes of finite type and with $\odr^\ubul$ being the de Rham complex of smooth forms. However, it is known that every finite CW-complex is homotopy equivalent to an open submanifold of an Euclidean space (see e.g. \cite[Corollary A.10]{h}). 
\end{proof}


\section{Betti-de Rham sets of 1-Hodge structures}
In this section, we define Betti-de Rham sets in a more general setting, which combines the affine Betti-de Rham sets and the Betti-de Rham sets of Simpson \cite{si}. They are analytic subsets of some abelian complex Lie groups. We will prove that any irreducible Betti-de Rham set is a translate of a Lie subgroup. 

Let $X$ be a quasi-compact K\"ahler manifold. According to \cite{deligne}, the group $H^1(X, \zz)$ admits a mixed Hodge structure of type $\{(1,0), (0,1), (1,1)\}$. Such a Hodge structure is called a 1-Hodge structure in \cite{a}. More precisely, a 1-Hodge structure consists of the following data:
\begin{enumerate}
\item A finitely generated abelian group $\Lambda$;
\item A subspace $W=W_1\subset \Lambda_{\bQ}$;
\item A subspace $F=F^1\subset \Lambda_{\bC}$;
\end{enumerate}
such that $W_\bC=(W_\bC\cap F)\oplus (W_\bC\cap \bar{F})$ and $\Lambda_\bC=W_\bC+F$. For simplicity, we will assume that the group $\Lambda$ is torsion free in this section. 

A subgroup $\Lambda'\subset\Lambda$ defines a sub 1-Hodge structure, if $\Lambda/\Lambda'$ is torsion free and $(\Lambda', W', F')$ is a 1-Hodge structure where $W'=W\cap \Lambda'_\bQ$, $F'=F\cap \Lambda'_\bC$. It is easy to check that given a sub 1-Hodge structure $(\Lambda', W', F')$ of $(\Lambda, W, F)$, the quotients $(\Lambda/\Lambda', W/W', F/F')$ has a natural 1-Hodge structure. In fact, this also follows from the general theory of Deligne \cite{deligne} that the mixed Hodge structures form an abelian category. 

Given a 1-Hodge structure $(\Lambda, W, F)$, let $\Lambda_0=\Lambda\cap W$. As the usual convention for 1-Hodge structures, we let $H^{0,1}=W_\bC\cap \bar{F}$, $H^{1, 0}=W_\bC\cap F$ and $H^{1,1}=\Lambda_\bC/W_\bC$. The complex Lie groups $\Lambda_\bC/\Lambda_0$ and $\Lambda_\bC/\Lambda$ are analogs of de Rham and Betti moduli spaces. The natural map $\rh: \Lambda_\bC/\Lambda_0\to \Lambda_\bC/\Lambda$, defined by taking further quotient, is the analog of the Riemann-Hilbert correspondence map. We call this map $\rh$ the Hodge-theoretic Riemann-Hilbert map. In general, $\Lambda_\bC/\Lambda_0$ will not have a natural algebraic variety structure. However, we will define a class of closed analytic subsets of $\Lambda_\bC/\Lambda_0$, which will define an analog of algebraic Zariski topology on $\Lambda_\bC/\Lambda_0$. In particular, when the 1-Hodge structure $(\Lambda, W, F)$ is the 1-Hodge structure of a smooth quasi-projective variety $X$, this topology will be same as the algebraic Zariski topology on the de Rham moduli space (see Section 4 for the latter).

Since $W_\bC=(W_\bC\cap F)\oplus (W_\bC\cap \bar{F})$ and since $\Lambda_\bC=W_\bC+F$, we have canonical isomorphisms $H^{0,1}\cong W_\bC/(W_\bC\cap F)\cong \Lambda_\bC/F$. Thus we have a canonical surjective map $p: \Lambda_\bC\to H^{0,1}$, whose kernel is equal to $F$. Moreover, $p(\Lambda_0)\cong \Lambda_0$ is a full lattice in $H^{0,1}$ over $\bR$. Let $T=H^{0, 1}/p(\Lambda_0)$. Then $T$ is a compact complex torus, and it is the analog of $\pic^0(\bar{X})$, where $\bar{X}$ is a good compactification of $X$. 

Let $\bar{p}: \Lambda_\bC/\Lambda_0\to H^{0,1}/p(\Lambda_0)=T$ be the natural map induced by $p$. Then $\bar{p}$ is an affine bundle map, with fiber isomorphic to $F$. More precisely, $\bar{p}: \Lambda_\bC/\Lambda_0\to T$ is a principal $F$ bundle map. Fiberwise, we can take a compactification $F\subset \bP^m$, where $m=\dim F$. We can globalize this compactification to get a compactification $\overline{\Lambda_\bC/\Lambda_0}$, which is $\bP^m$-bundle over T. More precisely, given a complex affine space $A$, we can define a vector space $V_A$ as follows, which we call the linearization of $A$. The vector space $V_A$ is generated by elements of $A$ subject to the relation that $\lambda a+(1-\lambda) b= c$ (as relation in $V_A$) for elements $a, b, c\in A$ with $\lambda a+(1-\lambda) b= c$ (as relation in $A$). Then the projective bundle $\overline{\Lambda_\bC/\Lambda_0}$ over $T$ is obtained first by a fiberwise linearization of the affine bundle $\bar{p}: \Lambda_\bC/\Lambda_0\to H^{0,1}/p(\Lambda_0)=T$, then taking fiberwise projectivization. 

\begin{defn}\label{defbdr}
Given a 1-Hodge structure $(\Lambda, W, F)$, recall that $\Lambda_0$ is defined by $\Lambda_0=\Lambda\cap W$. We call the abelian complex Lie group $\Lambda_\bC/\Lambda$ (resp. $\Lambda_\bC/\Lambda_0$)  the \textbf{Betti} (resp. \textbf{de Rham}) \textbf{moduli space associated to the 1-Hodge structure} $(\Lambda, W, F)$. 

We have defined a canonical compactification $\overline{\Lambda_\bC/\Lambda_0}$ of $\Lambda_\bC/\Lambda_0$, which is a $\bP^m$-bundle over the compact torus $T$. An analytic set of $\Lambda_\bC/\Lambda_0$ is called \textbf{de Rham closed}, if its Euclidean topological closure in $\overline{\Lambda_\bC/\Lambda_0}$ is also analytically closed. Clearly, the de Rham closedness defines a topology on $H/\Lambda_0$. We call this topology the \textbf{de Rham topology}. 

We will call an analytic set of $\Lambda_\bC/\Lambda$ \textbf{Betti closed}, if it is closed in the usual algebraic Zariski topology of $\Lambda_\bC/\Lambda$ via the identification $\Lambda_\bC/\Lambda\cong (\cc^*)^n$, where $n=\rank(\Lambda)$. That is, the \textbf{Betti topology} on $\Lambda_\bC/\Lambda$ corresponds to the usual algebraic Zariski topology on $(\cc^*)^n$. 

An irreducible algebraic set $S$ of $\Lambda_\bC/\Lambda$ is called an \textbf{irreducible Betti-de Rham set}, if there exists a de Rham closed set $R\subset \Lambda_\bC/\Lambda_0$ of the same dimension as $S$ such that $\rh(R)\subset S$. An algebraic set of $\Lambda_\bC/\Lambda$ is called a \textbf{Betti-de Rham set}, if it is a finite union of irreducible Betti-de Rham sets. 
\end{defn}

\begin{rmk}\label{bimeromorphic}
Let $\overline{\Lambda_\bC/\Lambda_0}'$ be another compactification of $\Lambda_\bC/\Lambda_0$ such that there is a bimeromorphic map between $\overline{\Lambda_\bC/\Lambda_0}'$ and $\overline{\Lambda_\bC/\Lambda_0}$ extending the identity map of $\Lambda_\bC/\Lambda_0$. Then an analytic set of $\Lambda_\bC/\Lambda_0$ is de Rham closed if and only if its Euclidean closure in $\overline{\Lambda_\bC/\Lambda_0}'$ is analytically closed. 
\end{rmk}

Let $(\Lambda', W', F')$ be a sub 1-Hodge structure of $(\Lambda, W, F)$. Then the image of the natural embedding $\Lambda'_\bC/\Lambda'\to \Lambda_\bC/\Lambda$ is an irreducible Betti-de Rham set. We will show the converse is also true up to a translate. 

\begin{theorem}\label{structurebdr}
Let $S$ be an irreducible Betti-de Rham set in $\Lambda_\bC/\Lambda$ as defined above. Then $S$ is a translate of a subtorus $\hat{S}$ of $\Lambda_\bC/\Lambda$. Moreover, $\hat{S}$ is defined by a sub 1-Hodge structure of $(\Lambda, W, F)$. 
\end{theorem}

\begin{proof}

The following commutative diagram of abelian complex Lie groups is the quasi-compact K\"ahler analog of the diagram in \cite{bw1}. 
\begin{equation}\label{ses}
\xymatrix{
0\ar[r] &(W_\bC/\Lambda_0)_\dr\ar[r]\ar[d]^{\rh} &\Lambda_{\mathbf{C}}/\Lambda_0\ar[d]^{\rh}\ar[r]^{\res} &H^{1,1}\ar[d]^{\rh}\ar[r] &0\\
0\ar[r] &(W_\bC/\Lambda_0)_\betti\ar[r] &\Lambda_\bC/\Lambda\ar[r]^{\ev\hspace{4mm}} &H^{1,1}/(\Lambda/\Lambda_0)\ar[r] &0
}
\end{equation}
Notice that both rows are exact, and the first vertical arrow is an isomorphism of complex Lie groups. Each complex Lie group in the diagram has two topologies. One is the Euclidean topology, which is induced from the complex manifold structures. All the arrows are continuous with respect to the Euclidean topology. There is another topology on each complex Lie group in the diagram, which is the analog of algebraic Zariski topology. In the second row, the second topology is the Betti topology. Each object in the second row is isomorphic to an affine torus. The Betti topology is the usual algebraic Zariski topology on affine tori. In the first row, the second topology is the de Rham topology. We have defined the de Rham topology on $\Lambda_\bC/\Lambda_0$. It follows from the definition that any 1-Hodge structure $(\Lambda, W, F)$ induces a 1-Hodge structure on $(\Lambda_0, W, F\cap W_\bC)$. With this induced 1-Hodge structure we can define the de Rham topology on $W_\bC/\Lambda_0$. The de Rham topology on the complex vector space $H^{1,1}$ is the algebraic Zariski topology on an affine space. To emphasize the second topology on $W_\bC/\Lambda_0$, we use the subscripts $DR$ and $B$ for the de Rham and the Betti topology, respectively. 


Now we prove the theorem in 3 steps. 

Case 1, $W=0$, that is when the first objects in the short exact sequences of (\ref{ses}) are trivial. In this case, $\Lambda_0=0$. In this case, $\Lambda_\bC/\Lambda_0\cong \Lambda_\bC$, and $\rh: \Lambda_\bC\to \Lambda_\bC/\Lambda$ is isomorphic to the exponential map. The definition of Betti-de Rham set is equivalent to the affine Betti-de Rham set. It follows from Proposition \ref{main0} that $S$ is a translate of a subtorus. When $W=0$, every subtorus of $\Lambda_\bC/\Lambda$ is defined by a sub 1-Hodge structure. Thus, the proposition follows. 

Case 2, $W=\Lambda_\bQ$, that is when the last objects in the short exact sequences of (\ref{ses}) are trivial. In this case, $W_\bC=\Lambda_\bC$, $\Lambda_0=\Lambda$ and $H^{1,1}=0$. The proposition is essentially proved in \cite[Theorem 3.1 (c)]{si}. In fact, Simpson proved this case with the assumption that $T=H^{0, 1}/p(\Lambda_0)$ is an abelian variety. However, the assumption is not necessary in the proof. The proof there is quite easy to follow, except the reference to \cite{si1}. So we will give a sketch of this part. 

We will prove there does not exist a codimension one irreducible Betti-de Rham set. In fact, let $F$ be any fiber of the principal affine bundle $\bar{p}: \Lambda_\bC/\Lambda\to T$, and let $G$ be a general affine line in $F$. Suppose there exists a codimension one irreducible Betti-de Rham set $S$. Considering $S$ as an algebraic subvariety of $\Lambda_\bC/\Lambda\cong (\cc^*)^n$, we can assume that $S$ is defined by an algebraic function $h$. It is a straightforward computation to check that the restriction of $h$ to $G$ has exponential growth. In other words, $h$ has essential singularity at infinity. Therefore, $h=0$ has infinitely many roots on $G$. This is a contradiction to $W$ being de Rham closed. 

For the rest of Simpson's proof in \cite{si} to work, we just need to check that the de Rham topology behaves in the same way as algebraic Zariski topology under projections. For example, let $s:\Lambda_\bC/\Lambda\times \Lambda_\bC/\Lambda\to \Lambda_\bC/\Lambda$ be the addition map, and let $Z_1$ and $Z_2$ be de Rham closed subsets of $\Lambda_\bC/\Lambda$. Then the closure of $s(Z_1\times Z_2)$, with respect to Euclidean topology, in $\Lambda_\bC/\Lambda$ is de Rham closed. In fact, this follows from Remark \ref{bimeromorphic} and the fact that the image of an analytic closed set under a proper map is still analytically closed. 

Case 3, the general case. Suppose $S$ is an irreducible Betti-de Rham set in the Betti moduli space $\Lambda_\bC/\Lambda$. Since being Betti-de Rham set is invariant under translations, we can assume that $S$ contains the the origin $0\in \Lambda_\bC/\Lambda$. Denote by $\Sigma S$ the Euclidean closure of the subgroup of $\Lambda_\bC/\Lambda$ generated by $S$. Then $\Sigma S$ is an irreducible Betti-de Rham set, which is a subgroup of $\Lambda_\bC/\Lambda$. It is easy to see that $\Sigma S$ is a subtorus defined by a sub 1-Hodge structure. We can also assume this by induction on the codimension of $S$. Replacing the original 1-Hodge structure $\Lambda$ by this sub 1-Hodge structure, we can assume that $S$ generates $\Lambda_\bC/\Lambda$ without loss of generality. For the rest of the proof, we will assume that $S$ contains the origin, and $S$ generates $\Lambda_\bC/\Lambda$, i.e., $\Sigma S=\Lambda_\bC/\Lambda$. 


As we have discussed in the beginning of this proof, the abelian group $\Lambda/\Lambda_0$ has a 1-Hodge structure of pure type (1,1). It follows easily from the definition of Betti-de Rham sets that $\overline{\ev(S)}$ is a Betti-de Rham set in the Betti moduli space $H^{1,1}/(\Lambda/\Lambda_0)$ associated to the 1-Hodge structure on $\Lambda/\Lambda_0$. Since $\overline{\ev(S)}$ contains the origin, by the conclusion of Case 1, $\overline{\ev(S)}$ is a translate of a subtorus in $H^{1,1}/(\Lambda/\Lambda_0)$. By the assumption that $S$ contains the origin and $S$ generates $\Lambda_\bC/\Lambda$, clearly $\overline{\ev(S)}=H^{1,1}/(\Lambda/\Lambda_0)$. 


Now, we use the short exact sequences (\ref{ses}) of Lie groups. We can consider the map $\Lambda_\bC/\Lambda_0\to H^{1,1}$ as a principal $W_\bC/\Lambda_0$ fiber bundle. Let $F$ be any fiber. 
By choosing an origin in $F$, there is an isomorphism $F\cong W_\bC/\Lambda_0$ of complex Lie groups. Since $H_0$ has a 1-Hodge structure, we can define Betti-de Rham sets of $F$ by Definition \ref{defbdr}. Moreover, the definition is independent of the choice of the origin, since the notion of Betti-de Rham sets is translation invariant. 
Now, clearly $S\cap F$ is a Betti-de Rham set of $F$. By the conclusion of Case 2, $S\cap F$ is a finite union of translates of Lie subgroups defined by sub 1-Hodge structures of $\Lambda_0$. There are at most countably many sub 1-Hodge structures of $\Lambda_0$. Therefore, for every fiber $F$ with $S\cap F\neq \emptyset$, $S\cap F$ is irreducible and it is a translate of the same subtorus $T$ of $W_\bC/\Lambda_0$. In other words, $S$ is invariant under the translation by elements in $T$. Moreover, subtorus $T$ is defined by a sub 1-Hodge structure $\Lambda_T$ of $\Lambda_0$. 

Since 1-Hodge structures form an abelian category, $\Lambda/\Lambda_T$ also has a natural 1-Hodge structure. Denote this 1-Hodge structure by $\Lambda'$ and denote $S/T$ by $S'$. By the above argument, $S'$ is an irreducible Betti-de Rham set in $\Lambda'_\bC/\Lambda'$. Moreover, the restriction of $\ev: \Lambda'_\bC/\Lambda'\to H'^{1,1}/(\Lambda'/\Lambda'_0)$ to $S'$ is dominant and generically finite. It suffices to show that $S'$ is a translate of a subtorus of $\Lambda'_\bC/\Lambda'$ defined by a sub 1-Hodge structure. Therefore, we are reduced to the case when the map $\ev|_S: S\to H^{1,1}/(\Lambda/\Lambda_0)$ is generically finite. 

We proceed using induction on $\dim_\bC H^{1,1}$. When $\dim_\bC H'^{1,1}=0$, the subspace $W'$ is a point, and we are done. When $\dim_\bC H'^{1,1}=1$, $S$ is a curve. The subsets $S$, $\overline{S+S}$, $\overline{S+S+S}$, $\ldots$, are all irreducible Betti-de Rham sets of $\Lambda_\bC/\Lambda$. By our assumption, the whole Betti moduli space $\Lambda_\bC/\Lambda$ will appear in the above sequence. Therefore, there exists a codimension one Betti-de Rham set in $\Lambda_\bC/\Lambda$ which maps dominantly under $\ev$ to $H^{1,1}/(\Lambda/\Lambda_0)$. Restricting such codimension one Betti-de Rham set to a general fiber $F$ of $\ev: \Lambda_C/\Lambda\to H^{1,1}/(\Lambda/\Lambda_0)$. By the earlier discussion, $F$ is isomorphic to the Betti moduli space associated to the pure 1-Hodge structure $\Lambda_0$ of weight one. By the argument in the proof of Case 2, codimension one Betti-de Rham subsets of $F$ do not exist, hence a contradiction. 

Suppose $\dim_\bC H^{1,1}>1$. Since the 1-Hodge structure on $\Lambda/\Lambda_0$ is pure of type (1,1), any surjective homomorphism $\Lambda/\Lambda_0\to \zz$ is a map of 1-Hodge structures, where we give the trivial type (1,1)-Hodge structure on $\zz$. Denote the kernel of $\Lambda/\Lambda_0\to \zz$ by $\Lambda'$. Clearly, $\Lambda'_0=\Lambda_0$. The intersection $S\cap \Lambda'/\Lambda_0$ in $\Lambda_\bC/\Lambda$ is a Betti-de Rham set in $\Lambda'/\Lambda_0$. By the induction hypothesis, $S\cap \Lambda'/\Lambda_0$ is a finite union of translates of subtori defined by sub 1-Hodge structures of $\Lambda'$. Notice that this conclusion holds for any $\Lambda'=\ker(\Lambda\to \zz)$, as long as the map $\Lambda\to \zz$ factors through the quotient map $\Lambda\to \Lambda/\Lambda_0$. One can easily deduce that near the origin $S$ is equal to the image of a linear subspace under the exponential map. Therefore, $S$ is a subtorus. 
\end{proof}

The following is an immediate consequence of the preceding theorem. 

\begin{cor}\label{bdrsets}
Let $S$ be a Betti-de Rham set in $\Lambda_\bC/\Lambda$. Then $S$ is a finite union of translates of subtori that are defined by a sub 1-Hodge structures of $(\Lambda, W, F)$. 
\end{cor}

\section{Cohomology jump loci of quasi-compact K\"ahler manifolds}\label{Kahler}
In this section we give the proof of Theorem \ref{main1}. Let $X$ be a quasi-compact K\"ahler manifold. We fix a good compactification $\bar{X}$, such that the boundary divisor $D=\bar{X}\setminus X$ is a normal crossing divisor in $\bar{X}$. Let $\mdr(\bar{X}/D)$ be the moduli space of logarithmic flat line bundles on $\bar{X}$ with poles along $D$. 
Let $\mb^0(X)$ and $\mdr^0(\bar{X}/D)$ respectively be the connected components of $\mb(X)$ and $\mdr(\bar{X}/D)$ containing the origin. The group $H^1(X, \zz)$ has a natural 1-Hodge structure. The Betti and de Rham moduli spaces associated to $H^1(X, \zz)$ are naturally isomorphic to $\mb^0(X)$ and $\mdr^0(\bar{X}/D)$. 

Theorem \ref{main1} will be reduced to the following statement.
\begin{prop}\label{prop_bdr}
Under the above notations, $\Sigma^i_k(X)\cap \mb^0(X)$ is a Betti-de Rham set in $\mb^0(X)$. 
\end{prop}


\begin{rmk}
Even though in general the moduli spaces of logarithmic flat bundles are only defined for quasi-projective manifolds, in the rank one case the definition easily extends to quasi-compact K\"ahler manifolds. The main reason is that rank one bundles are always stable. So there is no polarization and stability condition involved. In fact, one can show that $\mdr(\bar{X}/D)$ is a fine moduli space by constructing a universal family of flat bundles with logarithmic connections. This is very similar to the proof of the fact that the Picard group of a compact K\"ahler manifold is a fine moduli space. For the same reason, the moduli space $\mhod(\bar{X}/D)$ (as defined in \cite{si3}) of line bundles on $\bar{X}$ with $\lambda$-logarithmic connections with poles along $D$ is also well-defined for the pair $(\bar{X}, D)$. 
\end{rmk}

\begin{proof}[Proof of Proposition \ref{prop_bdr}]
By definition, the closed points of $\mdr(\bar{X}/D)$ parametrize the logarithmic flat line bundles $(E, \nabla)$, where $E$ is a line bundle on $\bar{X}$ and $\nabla: E\to E\otimes_{\sO_{\bar{X}}}\Omega_{\bar{X}}(\log D)$ is a logarithmic connection with poles along $D$. Notice that we do not put any assumption on the Chern classes of the underlying line bundle. So in general, $\mdr(\bar{X}/D)$ may have infinitely many connected components. In any case, there is a canonical covering map $RH: \mdr(\bar{X}/D)\to \mb(X)$ induced by taking restriction of the logarithmic flat bundle to $X$ and then taking the associated local system. We call this covering map the geometric Riemann-Hilbert map. Then the de Rham moduli space associated to the 1-Hodge structure $H^1(X, \zz)$ is isomorphic to $\mdr^0(\bar{X}/D)$. Moreover, the geometric Riemann-Hilbert map is isomorphic to the Hodge-theoretic Riemann-Hilbert map, that is the following diagram commutes,
$$
\xymatrix{
\Lambda_\bC/\Lambda_0\ar[r]^-{\cong}\ar[d]^{\rh}&\mdr^0(\bar{X}/D)\ar[d]^-{RH^0}\\
\Lambda_\bC/\Lambda\ar[r]^-{\cong}&\mb^0(X)
}
$$
where $\Lambda=H^1(X, \zz)$ with the natural 1-Hodge structure and $RH^0$ is the restriction of the geometric Riemann-Hilbert map $RH: \mdr(\bar{X}/D)\to \mb(X)$ to $\mdr^0(\bar{X}/D)$. 

Now, we need some conventions and results from \cite{bw1}. Define the de Rham cohomology jump loci to be 
$$\Sigma^i_k(\bar{X}/D)=\left\{(E, \nabla)\in \mdr(\bar{X}/D))|\dim \hh^i(\bar{X}, E\otimes \Omega^\ubul_{\bar{X}}(\log D))\geq k\right\}.$$
They are closed analytic subsets of $\mdr(\bar{X}/D)$. We will show that they are indeed de Rham closed subsets. We need the notation of $\lambda$-connections and their moduli space (see \cite{si3}). Denote by $\mhod(\bar{X}/D)$ the moduli space of line bundles on $\bar{W}$ with $\lambda$-logarithmic connections whose poles are along $D$. Then there is a projection $\pi: \mhod(\bar{X}/D)\to\aaa^1$, which is given by the parameter $\lambda$. Thus, $\pi^{-1}(1)\cong \mdr(\bar{X}/D)$. We can generalize the de Rham cohomology jump loci in $\mdr(\bar{X}/D)$ to the moduli space $\mhod(\bar{X}/D)$. Define
$$
\widetilde{\Sigma}^i_k(\bar{X}/D)=\left\{(E, \nabla)\in \mhod(\bar{X}/D))|\dim \hh^i(\bar{X}, E\otimes \Omega^\ubul_{\bar{X}}(\log D))\geq k\right\}.
$$
Notice that forgetting the connection, i.e. $(E, \nabla)\mapsto E$, defines two maps $\mdr(\bar{X}/D)\to \pic(\bar{X})$ and $\mhod(\bar{X}/D)\to \pic(\bar{X})$. Since the space of 1-connections on a line bundle is an affine space, and since the space of $\lambda$-connections on a line bundle is a vector space, $\mdr(\bar{X}/D)\to \pic(\bar{X})$ and $\mhod(\bar{X}/D)\to \pic(\bar{X})$ are affine bundle maps and vector bundle maps, respectively. Denote by $\pp(\mhod(\bar{X}/D)/\pic(\bar{X}))$ the projective bundle associated to the vector bundle $\mhod(\bar{X}/D)\to \pic(\bar{X})$. Then $\pp(\mhod(\bar{X}/D)/\pic(\bar{X}))$ is the projective compactification of the affine bundle $\mdr(\bar{X}/D)\to \pic(\bar{X})$. 

By definition, the cohomology jump locus $\widetilde{\Sigma}^i_k(\bar{X}/D)$ in $\mhod(\bar{X}/D)$ is preserved under the fiberwise $\bC^*$ multiplication. Denote the image of $\widetilde{\Sigma}^i_k(\bar{X}/D)$ in $\pp(\mhod(\bar{X}/D)/\pic(\bar{X}))$ by $\bar{\Sigma}^i_k(\bar{X}/D)$. Then $\bar{\Sigma}^i_k(\bar{X}/D)$ is a closed analytic subset of $\pp(\mhod(\bar{X}/D)/\pic(\bar{X}))$, and 
\begin{equation}\label{restriction}
\bar{\Sigma}^i_k(\bar{X}/D)\cap \mdr(\bar{X}/D)=\Sigma^i_k(\bar{X}/D)
\end{equation}
where we consider $\pp(\mhod(\bar{X}/D)/\pic(\bar{X}))$ as the fiberwise projective compactification of $\mdr(\bar{X}/D)$. 

Now, it follows from the definition of the 1-Hodge structure on $H^1(X)$ that the affine bundle map $\mdr^0(\bar{X}/D)\to \pic(\bar{X})$ is isomorphic to the affine bundle map $\bar{p}: \Lambda_\bC/\Lambda_0\to H^{0,1}/p(\Lambda_0)$ defined in the last paragraph before Definition \ref{defbdr}, where $\Lambda=H^1(X, \zz)$ with the natural 1-Hodge structure. It follows from the definition of de Rham subsets and the equality (\ref{restriction}) that $\Sigma^i_k(\bar{X}/D)\cap \mdr^0(\bar{X}/D)$ is a de Rham subset of $\Lambda_\bC/\Lambda_0$ by identifying $\mdr^0(\bar{X}/D)$ and $\Lambda_\bC/\Lambda_0$. Let $\mdr^1(\bar{X}/D)$ be another connected component of $\mdr(\bar{X}/D)$ whose image under the map $RH: \mdr(\bar{X}/D)\to \mb(X)$ is the component $\mb^0(X)$ containing the origin $\mathbf{1}$. Then by choosing a point of $RH^{-1}(\mathbf{1})\cap \mdr^1(\bar{X}/D)$ as an origin, we can construct an isomorphism between $\mdr^1(\bar{X}/D)$ and $\mdr^0(\bar{X}/D)$. Thus we obtain an isomorphism $\mdr^1(\bar{X}/D)\cong \Lambda_\bC/\Lambda_0$. Applying the same argument, we can show that $\Sigma^i_k(\bar{X}/D)\cap \mdr^1(\bar{X}/D)$ is a de Rham closed subset of $\Lambda_\bC/\Lambda_0$ via the above isomorphism. 

Define the bad locus $BL\subset \mdr(\bar{X}/D)$ to be the locus where one of the residues of $\nabla$ is a positive integer. By a theorem of Deligne \cite[II, 6.10]{d1} (see also \cite{bw1}), we have an equality
$$RH^{-1}(\Sigma^i_k(X))\setminus BL=\Sigma^i_k(\bar{X}/D)\setminus BL.$$
Given any irreducible component $S$ of $\Sigma^i_k(X)\cap \mb^0(X)$, it is proved in \cite{bw1} that there exists a connected component $\mdr^1(\bar{X}/D)$ of $\mdr(\bar{X}/D)$ such that $RH^{-1}(S)\cap \mdr^1(\bar{X}/D)$ is not contained in $BL$. In fact, take any point $\rho\in S$, and denote the Deligne extension of the local system $L_\rho$ by $(E_\rho, \nabla_\rho)$. Then the connected component of $\mdr(\bar{X}/D)$ containing $(E_\rho, \nabla_\rho)$ will work. Since the map $\mdr^1(\bar{X}/D)\to \mb^0(X)$ induced by $RH$ is a covering map, there exists an irreducible component $R$ of $\Sigma^i_k(\bar{X}/D)\cap \mdr^1(\bar{X}/D)\to \mb^0(X)$, such that $R$ has the same dimension as $S$ and $S=RH(R)$. Therefore, via the isomorphism $\mb^0(X)\cong \Lambda_\bC/\Lambda$, where $\Lambda=H^1(X, \zz)$, $S$ is a Betti-de Rham set. 

We have shown that every irreducible component of $\Sigma^i_k(X)\cap \mb^0(X)$ is a Betti-de Rham set. Thus, $\Sigma^i_k(X)\cap \mb^0(X)$ is a Betti-de Rham set. 
\end{proof}

\begin{proof}[Proof of Theorem \ref{main1}]
By Theorem \ref{structurebdr}, $\Sigma^i_k(X)\cap \mb^0(X)$ is a finite union of translates of a subtori. In particular, we have proved Theorem \ref{main1}, if $H_1(X, \zz)$ has no torsion. 

In general, since a finite cover of a quasi-compact K\"ahler manifold also has the structure of a quasi-compact K\"ahler manifold, it follows from a standard argument about cohomology jump loci of covering spaces (e.g. the last paragraph in the proof of \cite[Theorem 1.3]{w}) that $\Sigma^i_k(X)$ is a finite union of translated subtori. This finishes the proof of Theorem \ref{main1}.
\end{proof}

\end{document}